\documentclass[letterpaper, 10 pt, conference]{IEEEtran}  
\makeatletter\let\proof\@undefined\let\endproof\@undefined\makeatother
\IEEEoverridecommandlockouts                              

\usepackage{graphics} 
\usepackage{epsfig} 
\usepackage{amsmath,amssymb,amsthm} 
\newtheorem{theorem}{Theorem}
\newtheorem{lemma}[theorem]{Lemma}
\newtheorem{corollary}[theorem]{Corollary}
\newtheorem{prop}[theorem]{Proposition}

\usepackage{verbatim}
\usepackage{graphicx} 
\usepackage{epstopdf}
\usepackage{color,hyperref}
\definecolor{darkblue}{rgb}{0.0,0.0,0.3}
\hypersetup{colorlinks,breaklinks,linkcolor=darkblue,urlcolor=darkblue,anchorcolor=darkblue,citecolor=darkblue}
\usepackage{cite}

\usepackage{lipsum}
\makeatletter
\def\footnoterule{\relax%
	\kern-5pt
	\hbox to \columnwidth{\hfill\vrule width .9\columnwidth height 0.4pt\hfill}
	\kern4.6pt}
\makeatother

\makeatother
\title{\LARGE \bf  On Feedback Control in Kelly Betting: An  Approximation Approach
}

\author{\large Chung-Han Hsieh$^{*}$ 
	\thanks{\hskip -10pt ${}^*$Chung-Han Hsieh was with Department of Electrical and Computer Engineering, University of Wisconsin, Madison, WI 53706, USA. E-mail: \href{mailto:chunghan.hsieh@wisc.edu}{chunghan.hsieh@wisc.edu}. 	
		The results in this paper was evolved from a technical note for a  research project when the author studied in the Department of Mathematics, University of Wisconsin, Madison.
		The author would like to thank his dissertation advisor Professor B. Ross Barmish for discussing this line of research with the author.
	}
}
\begin{document}

	\maketitle
	
	
	\begin{abstract}
		In this paper, we consider a simple discrete-time optimal betting problem using the celebrated \textit{Kelly criterion}, which calls for maximization of the expected logarithmic growth of wealth.  
		While the classical Kelly betting problem can be solved via standard concave programming technique, an alternative but attractive approach is to invoke a Taylor-based approximation, which recasts the problem into quadratic programming and obtain the closed-form approximate solution. 
		The focal point of this paper is to fill some voids  in the existing results by providing some interesting properties when such an approximate solution is used. 
		Specifically, the best achievable betting performance, positivity of expected cumulative gain or loss and its associated variance, expected growth property, variance of logarithmic growth, and results related to the so-called \textit{survivability} (no bankruptcy) are provided.     
	\end{abstract}


	\section {Introduction}
	\label{Section: Introduction}
	Betting based on the celebrated Kelly criterion  \cite{Kelly_1956}, a prescription for optimal resource apportionment during favorable gambling games, has received a considerable attention in the literature; e.g., see \cite{Thorp_2006,Garlappi_Skoulakis_2011,Nekrasov_2014,Nekrasov_2014_book,Rising_Wyner_2012, MacLean_Thorp_Ziemba_2011,Cover_Thomas_2012,Maclean_Thorp_Ziemba_2010}. 
	Our focal point for this paper is to examine the maximization problem using Taylor-based approximation approach, which is frequently used in finance literature; e.g., see~\mbox{\cite{Thorp_2006,Garlappi_Skoulakis_2011,Nekrasov_2014,Nekrasov_2014_book,Rising_Wyner_2012}}. 
	It is well-known that this approximation-based method can lead to a solution which provides a certain insight on the risk-return tradeoffs and already achieved some successes in several empirical studies; e.g., see~\mbox{\cite{Nekrasov_2014,Rising_Wyner_2012,Nekrasov_2014_book}} and \cite{Pulley_1983}.
	
	In this regard, our aim in this paper is to fill the voids  in the existing results by exploring the properties of the approximate optimum. Several technical results such as best achievable upper bound performance, positivity of expected cumulative gain or loss and it associated variance, and results related to  \textit{survivability}; i.e., no-bankruptcy, in betting are~provided. 
	We should note here that the survivability issue is indeed closely related to the positivity issue of a state in system theory; e.g., see \cite{Hsieh_Gubner_Barmish_TAC2019}.
	
	To complete this brief overview, we mention a sampling of more recent work in the theory of Kelly Betting; e.g., see  \cite{Hsieh_Barmish_2015,Hsieh_Barmish_Gubner_2016,Hsieh_Gubner_Barmish_CDC2018,Hsieh_Barmish_Gubner_ACC2018},  an interesting application to option trading~\cite{Wu_Chung_2018}, and a rather comprehensive survey~\cite{MacLean_Thorp_Ziemba_2011} covering many of the most important papers. 
	%
	%

	\section{Problem Formulation}
	\label{Section 2: Problem Formulation}
	For $k=0,1,2,...\,$, let $X(k)$ be the \textit{returns} specified by the ``house" at stage $k$. We assume that the returns are bounded; i.e.,~\mbox{$X_{\min} \leq X(k) \leq X_{\max}$}
	with $X_{\min}$ and $X_{\max}$ being points in the support and satisfying
	\mbox{$
	X_{\min} < 0 < X_{\max}.
	$}
	In the sections to follow, we assume further that the random variables~$X(k)$ are independent and identically distributed~(i.i.d.). 
	
	\subsection{Betting Function and Account Value Dynamics}
	For stage $k=0,1,\ldots, N$, let $V(k)$ be the account value at stage $k$ and define a mapping \mbox{$u: \mathbb{N}\cup\{0\} \to \mathbb{R}$} to be the \textit{betting function} satisfying 
	$$
	u(k) := KV(k)
	$$
	where~\mbox{$K \in \mathcal{K} \subset \mathbb{R}$} with $\mathcal{K}$ being an  interval constraint on~$K$ which captures some practical betting restrictions\footnote{For example, as seen in Lemma~\ref{lemma: Survival Lemma} in Section~\ref{Section 2: Problem Formulation}, by taking \mbox{$\mathcal{K} := (-1/X_{\max}, 1/|X_{\min}|)$}, then we assure  \textit{survivability} (no-bankruptcy); i.e., $V(k)>0$ for all $k$ and all sample paths $X(0),X(1),\ldots, X(k-1)$.  As a second example, if $\mathcal{K} := [-1,1]$ which corresponds to the so-called {\it cash-financed} condition in finance. That is, $|u(k)| \leq V(k)$ for all $k$. }
	That is,   at stage $k$, the betting function $u(k)$ is proportion of the \textit{account value}~$V(k)$ being invested. 
	Note here that our formulation can be viewed as a game with two players: the \textit{bettor} and the \textit{house}. In this point of view, the negative~$K$ may be interpreted as betting on the ``other" side of the game; i.e., taking the role of the~house.  
	With initial account value~\mbox{$V(0)>0$}, the account value dynamics is determined by the stochastic recursive equation
	$$
	V(k+1) 	 = V(k) + u(k)X(k).
	$$
	Thus, the account value at terminal stage $N$ is readily  obtained as follows
	\[
	V\left( N \right) = \prod\limits_{k = 0}^{N-1} {\left( {1 + KX(k)} \right)} {V(0)}.
	\]

	\subsection{Feedback Control System Point of View}
	Throughout this paper, the approach we take  involves a  control-theoretic point of view. 
	In this regard, the language we use in this paper is consistent with a  growing body of the literature addressing finance problems but originating from the control community; e.g., see~\cite{Hsieh_Barmish_2015,Hsieh_Barmish_Gubner_2016,Hsieh_Gubner_Barmish_CDC2018,Hsieh_Barmish_Gubner_ACC2018}. Specifically, we view~$V(k)$ as the \textit{state} of a system with linear feedback control 
	\mbox{$u(k) = KV(k)$}
	where the constant~$K$ is  viewed as a {\it feedback gain}. 
	%
	
	\subsection{Survivability Considerations}\label{Survival Issue}
	One of the most important property that an solution~$K$ for the Kelly's maximization problem must hold is the so-called {\it survival} (no-bankruptcy) property. 
	That is, the feedback gain~$K$ must assure the account value $V(k) > 0$ for all~$k=0,1,...$. 
	The following result  characterizes the survival condition.

	\begin{lemma}[{\bf Survivability}]\label{lemma: Survival Lemma}
		{\it The survival condition holds; i.e., $V(k) > 0$ for all $k \geq 0$ and all sample paths $X(0),X(1),\ldots,X(k-1)$ if and only if the feedback gain~$K$~satisfies the inequality $$
			\frac{-1}{ X_{\max}} < K < \frac{1} {|X_{\min}|}.
			$$ }
	\end{lemma}
	
	\vspace{0mm}
	\begin{proof} To prove sufficiency, assume the inequality on  $K$ holds and we must show $V(k)>0$ for all $k$ and all sample paths. Note that $V(0)>0$. We proceed a proof   by induction. Assume $V(k)>0$ along any sample path $(X(0),X(1),\ldots,X(k-1))$, we must show~\mbox{$V(k+1)>0$}.  Note that for $w \in [X_{\min}, X_{\max}]$,
		\begin{align*}
		V(k+1) &= V(k)+X(k) KV(k) \\
		&\geq \min_{w \in [X_{\min}, X_{\max}]} \{V(k)+ wKV(k)  \}.
		\end{align*}
		Since the function to be minimized above is affine linear in~$w$,  its minimal value is achieved when \mbox{$w \in \{X_{\min},X_{\max}\}$}; see~\cite{barmish_1994}. 
		Therefore, to establish the desired survivability,  using the assumed inequality on~$K$ and the inductive hypothesis that $V(k)>0$,
		it follows that
		\[
		\frac{-1}{X_{\max}}V(k)< \underbrace{KV(k)}_{=u(k)} < \frac{1}{|X_{\min}|}V(k).
		\]
		Now, for $w=X_{\max}$, we have $V(k)+u(k)X_{\max}>0$. For $w=X_{\min}$, we again obtain $V(k)+u(k)X_{\min}> 0$. Thus, it follows that 
		$
		V(k+1)> 0. 
		$
		
		To prove necessity, assuming $V(k)>0$ for all $k\geq 0$ and all sample paths $X(0),X(1),\ldots, X(k-1)$, we have \mbox{$V(k+1)>0$}. We now show that the desired inequality on~$K$ holds. 
		Observe that
		\begin{align*}
		V(k+1) 
		&= (1+KX(k))V(k)>0
		\end{align*}
		which implies that
		$
		1+KX(k)>0
		$
		for all $k$.
		Thus, it follows that
		$
		1+KX_{\max}>0
		$
		and $1+KX_{\min}>0$, which leads to the desired inequality on~$K.$
	\end{proof}
	


	\subsection{Optimal Feedback Gain via Kelly Criterion}
	Having set up the control-theoretic framework, a subsequent important question we considered is as follows: How does one  choose a feedback gain $K$ so that the performance is regarded as ``optimal?" As mentioned in Section~\ref{Section: Introduction}, the Kelly criterion used in~\cite{Thorp_2006}  and~\cite{Nekrasov_2014}, suggests to maximize the expected logarithmic growth rate
	\[
	g(K) :=  \frac{1}{N} \mathbb{E} \left[ \log  \frac{V(N)}{V(0)} \right].
	\]
	With the aid of i.i.d. assumption on $X(k)$, we obtain
	\begin{align*}
	g(K) &= \frac{1}{N}\mathbb{E}\left[ {\log \left( {\prod\limits_{k = 0}^{N - 1} {\left( {1 + K X\left( k \right)} \right)} } \right)} \right]\\[.5ex] 
	&= \frac{1}{N} {\sum\limits_{k = 0}^{N - 1} \mathbb{E} [{\log \left( {1 + K X\left( k \right)} \; \right)]} } \\[.5ex]
	& =  \mathbb{E}[\;\log(1 + K X(0))\;].
	\end{align*}
	Our goal is to find an optimal feedback
	$$
	K \in \mathcal{K} := \bigg[ \frac{-1}{X_{\max}}, \frac{1}{X_{\min}} \bigg]
	$$ 
	such that  the expected logarithmic growth rate is maximized. Namely, we consider
	\[
	\max_{K \in \mathcal{K} } g(K)= \max_{K \in \mathcal{K}} \mathbb{E}[ \log(1 + K X(0))]
	\]
	and $K^* \in \mathcal{K}$ satisfying 
	$g^*:= g(K^*) = \max_{K \in \mathcal{K} } g(K)$ is called a (true) \textit{optimal feedback gain}.
	It is well-known that the optimization problem above forms a concave program since~$g(K)$ is concave in $K$ and the interval constraint $\mathcal{K}$ is, of course, convex; e.g., see also in~\cite{Hsieh_Barmish_Gubner_2016} and~\cite{Boyd_2004}  for a discussion of this topic. 
	In the sequel, we denote $g^*$ to be the (true) optimal performance.

	\section{Preliminary Characterization for Optimum}
	In practice, other than the constraint $K \in \mathcal{K}$, the so-called \textit{cash-financed} constraint $K \in [-1,1]$ is often imposed to assure  $|u(k)|\leq V(k)$ for all $k$. In this setting, mild assumptions on $X_{\min}$ and $X_{\max}$ lead to the following characterization of the optimal feedback gain. The result below can be viewed as an  extension of our Sufficiency Theorem stated in~\cite{Hsieh_Barmish_Gubner_ACC2018}, which includes the case for $K<0$.
	
	\begin{theorem}[\textbf{Characterizing Cash-Financed Optimum}]\label{thm: Cash-Fianced Optimum}
		Let $K \in [-1,1] \cap \mathcal{K}$ and assume the $X_{\max}$ and~$X_{\min}$ satisfy $-1<X_{\min}<0<X_{\max}<1$. Then the optimal feedback gain $K^*$ satisfies
		\[
		K^{*}= \begin{cases}
		1,& \mathbb{E}[1/(1+X(0))] \leq 1; \\
		-1, & \mathbb{E}[{1/(1-X(0))}] \leq 1.	\end{cases}
		\] 
	\end{theorem}
	
	\begin{proof} If $\mathbb{E}[1/(1+X(0))]\leq 1$, then \mbox{$\mathbb{E}[X(0)] \geq 0$}. In combination with the fact that the returns sequences are i.i.d., the optimal feedback gain $K^*$ must be nonnnegative. Hence, using the Sufficiency Theorem; see the result and the detailed proof in our prior work~\cite{Hsieh_Gubner_Barmish_CDC2018}, it follows that $K^*=1$. 
		
		Next, assuming that $\mathbb{E}[{1/(1-X(0))}] \leq 1$, we must show that $K^*=-1$. 
		To see this, we first note that $\mathbb{E}[{1/(1-X(0))}] \leq 1$, which implies that $\mathbb{E}[X(0)]\leq 0$. Hence, the optimal element must be nonpositive. 
		Therefore, it suffices to show that $g(K)$ is nonincreasing for~\mbox{$K \in [-1,0]$.} 
		Beginning with
		\begin{align*}
		\frac{d}{dK}g(K)
		&= \frac{d}{{dK}} \mathbb{E}\left[ {\log (1 + K X(k))} \right]
		\end{align*}
		and noting that $X(k)$ is bounded, results in measure theory, for example, see~\cite{Folland}, allow us to commute the differentiation and expectation operators above. Hence,
		\begin{align*}
		\frac{d}{{dK}}\mathbb{E}\left[ {\log (1 + K X(k))} \right]
		&= \mathbb{E}\left[ {\frac{{X(k)}}{{1 + KX(k)}}} \right].
		\end{align*}
		Now note that the inequality
		$
		\frac{z}{1+Kz}  \leq \frac{z}{1-z} = \frac{1}{1-z}-1
		$
		holds for all~$K \in [-1,0]$ and all $-1<z<1$. Hence, with the aid of the inequality and using the fact that the $X(k)$ are i.i.d. with $X_{\max}<1$, we~obtain
		\begin{align*}
		\frac{d}{{dK}}\mathbb{E}\left[ {\log (1 + KX(k))} \right]
		&\leq  \mathbb{E}\left[ {\frac{1}{1 - X(k)}} \right]-1\\[.5ex]
		&=  \mathbb{E}\left[ {\frac{1}{1 - X(0)}} \right]-1.
		\end{align*}
		Using the assumed inequality that
		$
		\mathbb{E}\left[ \frac{1}{1 - X(0)} \right] \leq 1,
		$
		it follows that
		$
		\frac{d}{dK} g(K) \le 0
		$
		which shows that $g(K)$ is nonincreasing in~$K$. Hence~$g(K)$ is maximized at $K =  -1$ and the proof is complete. \qedhere

	\end{proof}
	
	\textbf{Remarks:} Since the return sequences $X(k)$ are i.i.d., the two inequality conditions stated in the theorem above; i.e., $\mathbb{E}[1/(1+X(0))] \leq 1$ and \mbox{$\mathbb{E}[1/(1-X(0))]\leq 1$}, only depend on $X(0)$. In addition, these two inequalities indeed play a role to measure the attractiveness of a gamble, which are so-called \textit{sufficient attractiveness inequalities}; see our prior work in~\cite{Hsieh_Gubner_Barmish_CDC2018} and~\cite{Hsieh_Barmish_Gubner_ACC2018} for further discussion on this topic.
	In fact, Theorem~\ref{thm: Cash-Fianced Optimum} gives sufficient conditions under which $K^\ast$ is  analytically
	computed. 
	Except by a few special cases, it is not possible for finding the analytical solution in general. To this end, in the next section to follow, an approximation using  Taylor
	expansion is used.
	

	\section{A Taylor-based Approximation Approach} 
	\label{SECTION3:A Taylor-based Approximation Approach}
	Instead of solving the concave optimization problem described in Section~\ref{Section 2: Problem Formulation}, our goal in this paper is to study the expected logarithmic growth using a Taylor-based approximation approach.
	According to \cite{Nekrasov_2014}, it enjoys an arguably lower  computational complexity than that solves the Kelly problem in continuous-time setting. 
	When one considers stock trading scenario and historical stock return data is used, it is well-known that the second-order approximation is good enough; see~\cite{Pulley_1983}. 
	Moreover, perhaps the most important advantage of using such approximation is that this approach leads to a closed-form solution to the ``approximated" Kelly maximization problem, which provides a degree of insight into the risk-return tradeoffs. 
	To establish this, 
	according to \cite{Nekrasov_2014, Rising_Wyner_2012,Pulley_1983,Vajda_2006,Nekrasov_2014_book}, instead of working with $g(K)$, it is possible to uses the Taylor expansion on~$g(K)$ around~\mbox{$K=0$} to obtain an approximate quadratic function
	\[
	\mathbb{E} \left[ \log(1 + K X(0)) \right]  \approx  K\mathbb{E} \left[ X(0) \right] -\frac{1}{2} K^2 \mathbb{E}\left[ X^2(0) \right].
	\] 
	Subsequently, one then seeks feedback gain $K$ such~that 
	$$
	\max_K \left\{ K\mathbb{E} \left[ X(0) \right] -\frac{1}{2} K^2 \mathbb{E}\left[ X^2(0) \right] \right\}.
	$$
	Under this setting, one faces to solve a \textit{quadratic programming} problem.
	It is easy to see that the ``approximate" optimum, call it $K=K_{approx}^*$, is given by 
	\[
	K_{approx}^* :=  \frac{{\mathbb{E}\left[ X(0) \right]}}{{\mathbb{E} \left[ {{X^2(0)}} \right]}},
	\]
	which is the solution obtained in \cite{Nekrasov_2014} and \cite{Nekrasov_2014_book}.
	In the sequel, we shall often use shorthand notations to denote~\mbox{$\mu := \mathbb{E}[X(0)]$} and $\sigma^2 := {\rm var}(X(0))$ and write\footnote{ In practice, the information of $\mu$ and $\sigma$ may not be available. The simplest way is to estimate these two quantities based on the observations, say $X(k)=x_k$ and work with \textit{sample mean} $\mu_{_N}$ and \textit{sample variance} $\sigma_N$; i.e.,
		\mbox{$
		\mu_{_N} := \frac{1}{N} \sum_{k=0}^{N-1}x_k
		$}
		and
		\mbox{$
		\sigma_N^2 := \frac{1}{N-1} \sum_{k=0}^{N-1} (x_k-\mu_{_N})^2.
		$} 
		With the aid of i.i.d. assumption of $X(k)$ with common mean $\mu$ and common variance~$\sigma^2$, the strong law of large numbers implies $\mu_{_N} \to \mu$ and $\sigma_N^2 \to \sigma^2$ as $N\to \infty$; see \cite{Gubner_2006}.} 
	\[
	K_{approx}^* = \frac{\mu}{\mu^2+\sigma^2}.
	\]
	In theory, the $K_{approx}^*$ can take any value on~$\mathbb{R}$.\footnote{Some of the literature, based on empirical data support; e.g., see~\cite{Rising_Wyner_2012}, assume that $\mathbb{E}[X^2(0)] \approx {\rm var}(X(0))$. This leads to an alternative approximate solution
		\[
		\widetilde{K}_{approx} = \frac{\mathbb{E}[X(0)]}{{\rm var}(X(0))} = \frac{\mu}{\sigma^2}
		\]
		and coincides with the celebrated Merton's formula in continuous-time setting; e.g.,  see~\cite{Merton_1969}. In this paper, whenever the approximate solution is referred, we mean~$K_{approx}$. The $\widetilde{K}_{approx}$ is left in commentary.} This property can be readily interpreted using financial market language as follows:  $K_{approx}^* >1$ corresponds to the \textit{leverage} and $K_{approx}^* <0$ corresponds to \textit{short selling}.
	
	\subsection{Survival Conditions Revisited: An Example}
	The Survival Lemma in Section~\ref{Section 2: Problem Formulation} tells us that any feedback gain $K \in \mathcal{K}$ assures that $V(k)>0$ for all $k$.  It is natural to examine the approximate solution $K_{approx}^*$ and see if it meets the survival requirement. The following example indicates that this needs not  be the case in general. 
	
	%
	For example, let $k = 0,1,\ldots, N$, suppose  a coin-flipping gamble with returns $X(k)$ takes two distinct values as \mbox{$X(k) = -0.9$} with probability~$0.05$ or $X(k) = 0.2$ with probability $0.95$.  Then, it is readily seen that the corresponding approximation optimum is
	$
	K_{approx}^* \approx 1.84 .
	$
	However, the~$K_{approx}^*$ is not within the survival range; i.e.,
	$$
	K_{approx}^* \notin  \mathcal{K} = (-5, 1.111).
	$$
	Moreover, consider the worst case sample path; i.e.,~\mbox{$X(k) = -0.9$} for all $k \geq 0$, with $V(0) = 1$, it follows  that 
	\mbox{$
		V(1) = 1+K_{approx}^*(-0.9) \approx  -0.656 < 0
		$}
	which fails to survive and we see a single-stage ruin with  probability which fails to survive as $p=.05$. 
	
	
	
	\subsection{Simple Remedy for Survival Issues}
	While the approximate solution does not meet the survivability in almost-sure sense in general,  
	one can easily restrict it back to the range where the Survival Lemma asks for. For example, one approach is to introduce the \textit{saturation function}; i.e., we define
	\[
	K_{sat,s}^* := SAT_s \left [ K_{approx}^* \right]
	\]
	where $SAT_s[x]$ is given by
	\[
	SAT_s\left[ x \right] :=
	\begin{cases}
	\frac{-1}{X_{\max}} & x< \frac{-1}{X_{\max}};\\
	x & \frac{-1}{X_{\max}} \leq x \leq \frac{1}{|X_{\min}|};\\
	\frac{1}{|X_{\min}|} & x>\frac{1}{|X_{\min}|}
	\end{cases}
	\]
	and the subscript $s$ in the saturation function is used to emphasize the survivability as requested in Lemma~\ref{lemma: Survival Lemma}.
	
	Thus, with the aid of saturation, the approximated feedback is always within the upper and lower
	bounds of $[-1/X_{\max}, 1/|X
	_{\min}|].$ To this end, in the analysis to follow, we assume that $K_{approx}^* \in \mathcal{K}$.
	Of course, the above is not the only remedy; one can also consider the \textit{logistic function} to obtain a smooth saturation.
	\section{Betting Performance Analysis}
	\label{SECTION: Betting Performance Analysis}
	In this section, using the approximate solution, we now provide several technical results such as the Best Possible Performance, the Cumulative Gain or Loss Function and its expected value, and a probabilistic quantification.
	
	\subsection{Best Possible Betting Performance}
	We begin with providing an estimate for the \textit{best} possible performance when $
	K_{approx}^*$ is used.
	
	\begin{lemma}[{\bf Best Possible Performance}]\label{Lemma: Best Possible Performance}
		{\it For $K=K_{approx}^*$, we have 
			$$
			g(K_{approx}^*)  \leq  \log \left(1+ \frac{ \mu^2 }{\mu^2 + \sigma^2} \right) \leq \log 2.
			$$
		}
	\end{lemma}
	
	\vspace{0mm}
	\begin{proof}
		To establish the desired upper bound, we fix $K \in \mathcal{K}$ and apply Jensen's inequality to obtain
		\begin{align*}
		g(K) & = \mathbb{E} [\log(1+KX(0))] \\
		& \leq \log(1+K \mathbb{E}[X(0)])\\
		& = \log(1+K \mu).
		\end{align*}
		Substituting $K=K_{approx}^*$ into the inequality above, we obtain
		$$	
		g(K_{approx}^*) \leq  \log \left(1+ \frac{ \mu^2 }{\mu^2 + \sigma^2} \right).
		$$
		To complete the proof, we note that since $\mu^2/(\mu^2+\sigma^2) \leq 1$, it follows that
		$
		g(K) \leq \log 2. 
		$ 
	\end{proof}
	
	\textbf{Remarks:} $(i)$ The upper bound in the lemma above is achievable when $\sigma^2 =0$. In practice, this is possible if one enters a game with  $X(k)$ being the riskless returns with riskless rate $r > 0$; i.e.,~$
	X(k) := r > 0
	$ with probability one.
	Then the associated approximate solution becomes $K_{approx}^* = 1/r$ and one sees
	\[
	g(K_{approx}^*) = \mathbb{E} [\log(1+ (1/r)r)] = \log 2.
	\]
	$(ii)$ For $0 \leq \theta \leq 1$, it is also interesting to note that one can apply the Paley-Zygmund inequality to obtain~\mbox{$
	P( |X(0)|> \theta \mu) \geq (1-\theta)^2 K_{approx}^*
	$}
	which might be useful to estimate the behavior of the returns.
	The reader is referred to \cite{Paley_Zygmund_1932} for a detailed discussion on this topic.  $(iii)$ If one applies Merton's formula 
	$K=\widetilde{K}_{approx}^*$, then it is readily shown that the upper bound in the lemma above becomes 
	$
	g(\widetilde{K}_{approx}^*) \leq \log \left(1+ { \mu^2 }/{ \sigma^2} \right).
	$

	\subsection{Cumulative Gain or Loss}
	Given any betting strategy, it is often important for a gambler to know what is the expected cumulative gain or loss and its associated variance. 
	To this end, we  define the  \textit{cumulative gain or loss function}, call it $\mathcal G_K$, as follows: Let $\mathcal G_K: \mathbb{N} \to \mathbb{R}$ with
	$
	\mathcal{G}_K(N) := V(N) - V(0)
	$
	where the subscript $K$ on~$\mathcal{G}_K$ above is used to emphasize the dependence on feedback gain~$K$. 
	Then the expected cumulative gain or loss function, call it $ \overline{\mathcal{G}}_K$, is given~by 
	$
	\overline{\mathcal{G}}_K(N) := \mathbb{E}[\mathcal{G}_K(N) ]
	$
	and we are now ready to provide the results to follow.
	
	\begin{lemma}[{\bf Expected Cumulative Gain or Loss}]\label{Lemma: Expected Cumulative Gain or Loss}
		{\it Given any linear feedback $K \in \mathcal{K}$ and integer~\mbox{$N>0$}, the expected cumulative gain or loss function is given~by
			$$
			\overline{\mathcal{G}}_K(N) = \left( (1+K\mu)^N -1 \right)V\left( 0 \right),
			$$
			and hence, for $
			K=K_{approx}^*
			$,
			\[
			\overline{\mathcal{G}}_{K_{approx}^*}(N) = \left( {{{\left( {1 + \frac{{{\mu ^2}}}{{{\mu ^2} + {\sigma ^2}}}} \right)}^N} - 1} \right)V\left( 0 \right).
			\]
		}
	\end{lemma}
	
	\vspace{0mm}
	\begin{proof} With the aid of i.i.d. property of $X(k)$, it is readily verified that 
		\begin{align*}
		\overline{\mathcal{G}}_K(N) &= \mathbb{E}[V(N)] - V(0)  \\[.5ex]
		&= \mathbb{E}\left[ {\prod\limits_{k = 0}^{N - 1} {\left( {1 + KX(k)} \right)} V(0)} \right] - V(0) \hfill \\[.5ex]
		&= \left(  {\prod\limits_{k = 0}^{N - 1} {\left( {1 + K\mathbb{E}\left[ {X(0)} \right]} \right)} }  - 1 \right)V(0) \hfill \\[.5ex]
		&= \left( {{{\left( {1 + K\mu } \right)}^N} - 1} \right)V(0).  
		\end{align*} 
		To complete the proof, we substitute $K=K_{approx}^*$ into~$\overline{\mathcal{G}}_K(N)$ and $\overline{\mathcal{G}}_{K_{approx}^*}(N)$ is immediately obtained. 
	\end{proof}

	{\bf Remark:} 
	If one adopts the Merton's formula; i.e., 
	\mbox{$
		K=\widetilde{K}_{approx}^*,
		$}
	then 
	\[
	\overline{\mathcal{G}}_{\widetilde{K}_{approx}^*}(N) = \left( {{{\left( {1 + \frac{{{\mu ^2}}}{ \sigma ^2}} \right)}^N} - 1} \right)V\left( 0 \right).
	\]
	As seen in the corollary below, we can deduce more regarding the positivity of expected value of the cumulative gain or loss function.

	\begin{corollary}[\textbf{Positive Expectation Property}]{\it For~\mbox{$
				K=K_{approx}^*$}, the expected cumulative  gain or loss function is nonnegative; i.e.,
			$
			\overline{\mathcal{G}}_{K_{approx}^*} (N) \geq 0
			$ for all $N \geq 1$.
			Moreover, if $\mu \neq 0$, then~$\overline{\mathcal{G}}_{K_{approx}^*} (N) >0.$
		}
	\end{corollary}
	
	\begin{proof} The first statement is a simple consequence by the Lemma~\ref{Lemma: Expected Cumulative Gain or Loss}. That is, for $N \geq 1$,  we have 
		\[{{{\left( {1 + \frac{{{\mu ^2}}}{{{\mu ^2} + {\sigma ^2}}}} \right)}^N}} \geq 1
		\] 
		and note that the inequality above is strict if $\mu \neq 0$. Thus, it follows that
		\[
		\overline{\mathcal{G}}_{K_{approx}^*}(N) = \left( {{{\left( {1 + \frac{{{\mu ^2}}}{{{\mu ^2} + {\sigma ^2}}}} \right)}^N} - 1} \right)V\left( 0 \right) \geq 0
		\]
		and if $\mu \neq 0$, we see that  $\overline{\mathcal{G}}_{K_{approx}^*}(N)>0$ follows~immediately. 
	\end{proof}

	{\bf Remark:} The statement of the corollary above holds true for the Merton's formula $
	K=\widetilde{K}_{approx}^*
	$. Moreover, with the aid of the corollary, a  somewhat stronger result, recorded below, can be proven.

	\begin{theorem}[\textbf{Expected  Growth Property of Optimum}]{\it For~\mbox{$
				K=K_{approx}^*$}, the expected cumulative gain or loss function satisfies 
			$
			\overline{\mathcal{G}}_{K_{approx}^*} (N+1) \geq 	\overline{\mathcal{G}}_{K_{approx}^*} (N) \geq 0
			$ for all $N \geq 1$.
		}
	\end{theorem}
	
	\begin{proof}
		Observe that
		\[
		\overline{\mathcal{G}}_{K_{approx}^*} (N)  = \left( {{{\left( {1 + \frac{{{\mu ^2}}}{{{\mu ^2} + {\sigma ^2}}}} \right)}^N} - 1} \right)V\left( 0 \right).
		\] 	
		Since the term ${{{\left( {1 + \frac{{{\mu ^2}}}{{{\mu ^2} + {\sigma ^2}}}} \right)}^N} - 1}$ is increasing in $N$  and $V(0)>0$, it immediately follows that
		\[
		\overline{\mathcal{G}}_{K_{approx}^*} (N+1) \geq 	\overline{\mathcal{G}}_{K_{approx}^*} (N). \qedhere
		\]
	\end{proof}
	
	
	In finance, the \textit{variance} is a widely used risk metric; e.g., see~\cite{luenberger_2013}.
	Thus, it is interesting to study the variance of the cumulative gain or loss function induced by the approximate solution, call it ${\rm var}(\mathcal{G}_{K_{approx}^*}(N))$.  The following result gives a closed-form expression on it.

	
	\begin{lemma}[\textbf{Variance of Cumulative Gain or Loss}]
		{\it Given any linear feedback $K \in \mathcal{K}$ and integer~\mbox{$N \geq 1$},  the variance of the cumulative gain or loss function satisfies
			$$
			{\rm var}(\mathcal{G}_{K}(N)) = \left( (\mu_{_K}^2 + \sigma_{_K}^2 )^N - \mu_{_K}^{2N}\right)V^2(0).
			$$
			where 
			$\mu_{_K} :=  1 + K\mu$ and $\sigma_{_K} := K \sigma.$ In addition, for~\mbox{$K=K_{approx}^*$}, then
			{\footnotesize
				$$
				{\rm var}(\mathcal{G}_{K_{approx}^*}(N)) =\left(\left(\frac{4 \mu ^2+\sigma ^2}{\mu ^2+\sigma ^2}\right)^{N}-\left(\frac{2\mu ^2+\sigma ^2}{\mu ^2+\sigma ^2}\right)^{2 N}\right)V^2(0).
				$$}
		}
	\end{lemma}
	
	\begin{proof} Fix $N \geq 1$, begin by noting that
		$$
		{\rm var}(\mathcal{G}_K(N)) = \mathbb{E}\left[ {{V^2}\left( N \right)} \right] - {\left( {\mathbb{E}\left[ {V\left( N \right)} \right]} \right)^2}.
		$$
		Now, using the fact that $X(k)$ are i.i.d., we observe that  
		\begin{align*}
		{\mathbb{E}\left[ {V\left( N \right)} \right]} 
		&= {\left( {1 + K\mu} \right)}^N {V(0)}
		\end{align*}
		and hence ${\left( {\mathbb{E}\left[ {V\left( N \right)} \right]} \right)^2} = {\left( {1 + K\mu} \right)}^{2N}V^2(0).$ On the other hand, using the fact that $X(k)$ are i.i.d. again, we have 
		\begin{align*}
		{\mathbb{E}\left[ {V^2\left( N \right)} \right]} 
		&=\prod\limits_{k = 0}^{N-1}  \mathbb{E}\left[ {\left( {1 + KX(k)} \right)^2}  \right]V^2(0) \\
		&= \mathbb{E}\left[  1 + 2KX(0) + K^2X^2(0) \right]^N V^2(0) \\
		&= (  1 + 2K\mu+ K^2(\mu^2+\sigma^2) )^N V^2(0)\\
		&=(  (1 + K\mu)^2+K^2 \sigma^2 )^N V^2(0).
		\end{align*}
		Thus, we have
		$$
		{\rm var}(\mathcal{G}_{K_{approx}^*}(N)) = \left( (\mu_{_K}^2 + \sigma_{_K}^2 )^N - \mu_{_K}^{2N}\right)V^2(0).
		$$
		where $\mu_{_K} =  1 + K\mu$ and $\sigma_{_K} = K \sigma$. To complete the proof, substituting $K=K_{approx}^*$ into the equality above, a straightforward calculation leads~to
		the desired result. 
	\end{proof}
	
	{\bf Remark:} It is trivial to see that if $\sigma = 0$, then
	$
	{\rm var}(\mathcal{G}_{K}(N)) = 0.
	$

	\subsection{Variance of Logarithmic Growth}
	According to~\cite{luenberger_2013}, it is also useful to know the variance of expected log-growth since this quantity can be viewed as an additional risk metric, which is suitable for the Kelly's expected log-growth maximization framework. The following lemma summarizes the variance in a closed-form.
	
	\begin{lemma}[\textbf{Variance of Logarithmic Growth}]
		\textit{The variance of logarithmic growth is given by
			\[
			{\rm var}\bigg(\log \frac{V(N)}{V(0)}\bigg) = N\bigg(\mathbb{E}[  \log^2(1+KX(0))]- g^2(K)\bigg).
			\]
		}
	\end{lemma}
	
	\begin{proof}
		Observe that
		\[
		\log\frac{V(N)}{V(0)} = \sum_{k=0}^{N-1}\log(1+KX(k)).
		\]
		Hence, with the aid of i.i.d. of $X(k)$, it follows that
		\[
		{\rm var}\bigg( \sum_{k=0}^{N-1} Z_k\bigg) =  \sum_{k=0}^{N-1} {\rm var}(Z_k)
		\] where $Z_k:= \log(1+KX(k))$.
		In addition,  it is readily verified that for each $k=0,1,\ldots,N-1$,
		\begin{align*}
		{\rm var}(Z_k)&=\mathbb{E}[\log^2(1+KX(k))]-g^2(K).
		\end{align*}
		Therefore, a straightforward calculation leads to
		\begin{align*}
		{\rm var}\bigg(\log \frac{V(N)}{V(0)}\bigg) &= \sum_{k=0}^{N-1} \mathbb{E}[\log^2(1+KX(k))]-N g^2(K)\\
		&= N\bigg(\mathbb{E}[  \log^2(1+KX(0))]- g^2(K)\bigg)
		\end{align*}
		which is desired.
	\end{proof}

	\subsection{Performance Via Approximation: A Revisit}
	Henceforth, if \mbox{$K=K_{approx}^*$} is used, then~$g(K_{approx}^*)$ represents the corresponding expected logarithmic growth rate in wealth using the approximate optimum. Similarly, if~\mbox{$K=K^*$}, then $g(K^*)$ represents the expected logarithmic growth rate using true optimum. 
	The following proposition  estimates an upper bound for the difference between~$g(K_{approx}^*)$ and~$g(K^*)$.

	\begin{prop} {\it If $K^* \neq K_{approx}^*$ and both of them satisfies the condition stated in the Survival Lemma, then 
			\[
			0\leq g(K^*) - g(K_{approx}^*)\leq  \log \mathbb{E} \left[ {\frac{{1 + {K^*}X\left( 0 \right)}}{{1 + K_{approx}^*X\left( 0 \right)}}} \right] .
			\]
			Otherwise,
			$
			|g(K^*) - g(K_{approx}^*)|=0.
			$
		}
	\end{prop}
	
	\begin{proof} When $K^* = K_{approx}^*$, by definition of $g$, it is trivial to see that $
		|g(K^*) - g(K_{approx}^*)|=0.
		$
		To complete the proof, it suffices to show that the desired upper bound holds when $K^* \neq K_{approx}$.  Note that $g(K^*) \geq g(K_{approx}^*)$, hence, it is obvious that $0\leq g(K^*) - g(K_{approx}^*)$. On the other hand, using the fact that $K^*$ and $K_{approx}^*$ satisfies the condition stated in the Survival Lemma, it is readily seen that both $1+K^*X(0) > 0$ and $1+K_{approx}^*X(0)> 0$ for all admissible~$X_{\min} \leq X(0) \leq X_{\max}$ with probability one.  Now,  using Jensen's inequality on the logarithmic function, we obtain
		\begin{align*}
		{g\left( K^*\right) - g\left( {K_{approx}^*} \right)}  &= \mathbb{E} \left[ {\log \frac{{1 + {K^*}X\left( 0 \right)}}{{1 + K_{approx}^*X\left( 0 \right)}}} \right] \hfill \\[.5ex]
		&\leq \log \mathbb{E} \left[ {\frac{{1 + {K^*}X\left( 0 \right)}}{{1 + K_{approx}^*X\left( 0 \right)}}} \right].  \qedhere
		\end{align*} 
	\end{proof}
	
	{\bf Remarks:} It is worth noting that inside the upper bound of the performance difference $\mathbb{E} \left[ {\frac{{1 + {K^*}X\left( 0 \right)}}{{1 + K_{approx}^*X\left( 0 \right)}}} \right]$  is of the form of a linear-fractional function. Thus, if needed, one can carry out a next level estimate on the upper bound by invoking \textit{linear-fractional programming} technique. That is, one can  consider an optimization problem given by
	\begin{align*}
	&\max_{z} \frac{1+K^* z}{1+K_{approx}^*z}\\
	&\text{subject to} \begin{bmatrix}
	1 \\-1
	\end{bmatrix}z \leq \begin{bmatrix}
	X_{\max} \\-X_{\min}
	\end{bmatrix}.
	\end{align*}
	Then, using Charnes-Cooper transformation; see \cite{Boyd_2004} and~\cite{Charnes_Cooper_1962}, one can readily recast above into a linear programming problem then solve the unknown $z$ above in a very efficient way.

	\section{Conclusion and Future Works}
	In this paper, we examined several properties for the existing approximate solution using the celebrated maximization of expected logarithmic growth as performance metric. We see that the solution indeed provides a certain degree of insights on risk-return tradeoffs. Several technical results such as best achievable upper bound performance, positivity of expected cumulative  gain or loss and results related to survivability are provided. 
	
	Regarding further research, one possible continuation  is to generalize  the betting  function  to include  time-varying feedback gain $K$; i.e., $K=K(k)$. Another immediate direction for future research would be  to extend our framework  to the stock trading scenario, which involves multiple stocks into considerations.  Finally, since the Kelly criterion requires the bettor to know the distribution of returns, it is natural to ask  what if the underlying distributions are not trustworthy? What is the associated performance lead by approximate solution? One possible approach is to formulate a Kelly problem involving uncertain distributional considerations; e.g., see~\cite{Rujeerapaiboon_Kuhn_Wiesemann_2015}.

	
	\bibliographystyle{ieeetr}
	\bibliography{Note_Kelly_Approximation_refs}


\end{document}